\UseRawInputEncoding
\documentclass[12pt, twoside, reqno]{article}
\usepackage{amssymb}
\usepackage{amsmath}
\usepackage{amsthm}
\usepackage{color}
\usepackage[colorlinks,linkcolor=black, citecolor=black]{hyperref}
\usepackage{xparse}
\usepackage{mathrsfs}
\usepackage{verbatim}
\usepackage{indentfirst}
\usepackage{fancyhdr}

\def\rr{{\mathbb R}}
\def\rn{{\mathbb{R}^n}}

\def\zz{{\mathbb Z}}

\def\nn{{\mathbb N}}

\def\fz{\infty }
\def\az{\alpha}

\def\lz{\lambda}

\def\lf{\left}
\def\r{\right}

\def\hs{\hspace{0.25cm}}

\def\loc{{\mathop\mathrm{\,loc\,}}}
\def\supp{\mathop\mathrm{\,supp\,}}

\def\XXint#1#2#3{{\setbox0=\hbox{$#1{#2#3}{\int}$ }
\vcenter{\hbox{$#2#3$ }}\kern-.6\wd0}}

\def\B{{B}_{p,q}^\az}

\def\({\left(}
\def \){ \right)}

\def\lz{{\lambda}}

\newtheorem{theorem}{Theorem}[section]
\newtheorem{corollary}{Corollary}[section]
\newtheorem{lemma}{Lemma}[section]

\newtheorem{definition}{Definition}[section]

\newtheorem{proposition}{Proposition}[section]
\newtheorem{remark}{Remark}[section]

\renewcommand{\appendix}{\par
   \setcounter{section}{0}%
   \setcounter{subsection}{0}%
   \csname appendixmore\endcsname
 }

\numberwithin{equation}{section}

\title{\bf\Large Mixed-norm Herz-slice Spaces and Their Applications}
\author{Lihua Zhang, Jiang Zhou}

\begin{document}

\date{}
\maketitle
\vspace{-0.7cm}

\begin{center}
\begin{minipage}{13.5cm}
 {\bf Abstract}\quad 
We introduce mixed-norm Herz-slice spaces unifying classical Herz spaces and mixed-norm slice spaces, establish dual spaces and the block decomposition, and prove that the boundedness of Hardy–Littlewood maximal operator on mixed-norm Herz-slice spaces.
\end{minipage}
\end{center}

\begin{center}
\begin{minipage}{13.5cm}
\end{minipage}
\end{center}

\begin{center}
\begin{minipage}{13.5cm}
\end{minipage}
\end{center}

\footnotetext[1]{\begin{flushleft}
Lihua Zhang,\\
College of Mathematics and System Science,  Xinjiang University, 
No. 777, Huarui Street, Shuimogou District, Urumqi City, Xinjiang Uygur Autonomous Region, China\\ 
e-mail: hanyehua666@163.com
\end{flushleft}}
\footnotetext[2]{\begin{flushleft}
Jiang Zhou, \\
College of Mathematics and System Science,  Xinjiang University, 
No. 777, Huarui Street, Shuimogou District, Urumqi City, Xinjiang Uygur Autonomous Region, China \\
e-mail: zhoujiang@xju.edu.cn            
          \end{flushleft}}

\section{Introduction\label{s1}}
In 1964, Beurling \cite{BA} originally introduced the Herz space $A_{u}(\rn)$, which is the original version of the Herz space of non-homogeneous type. In 1968, the Herz space $K_{u}(\rn)$ has been studied systematically by Herz \cite{Herz}. In the 1990's, the homogeneous Herz space $(\dot K_u^{\beta,s})(\rn)$ and the non-homogeneous Herz space $(K_q^{\alpha,p})(\rn)$ are introduced by Lu and Yang \cite{LY}. In recent years, the homogeneous Herz space $(\dot K_u^{\beta,s})(\rn)$ and the non-homogeneous Herz space $( K_u^{\beta,s})(\rn)$ was investigated in harmonic analysis, see, \cite{CN,NZ,ST}. For more research about Herz spaces in PDE and harmonic analysis, we refer to \cite{AS,HK,MWY} and so on.

In 1961, the mixed-norm Lebesgue space $L^{\vec{u}}\left(\mathbb{R}^n\right)$ with $\vec{u}=(u_{1},\ \ldots, u_{n})\in(0,\infty]^{n}$
was researched by Benedek and Panzone \cite{BP}, which can go back to \cite{H} in 1960. Since function spaces with mixed norms have a broader purpose on PDE \cite{AI,KN}, more people renewed interest. In recent years, to gain the convergence of the Fourier transform,
in 2021, Huang, Weisz, Yang, and Yuan \cite{HFDY} introduced the mixed-Herz space $\dot{E}_{\vec{u}'}^{\ast}(\rn)$.
In 2021, Wei \cite{WM} established the boundedness of the Hardy--Littlewood maximal operators on mixed-Herz space $(\dot K_{\vec{u}}^{\beta,s})(\rn)$.
In 2022, Zhang and Zhou \cite{ZZM} introduced the mixed-norm amalgam space $(L^{\vec{u}},L^{\vec{v}})(\mathbb{R}^{n})$, and the boundedness of the Hardy--Littlewood maximal operators on the mixed-norm amalgam space $(L^{\vec{u}},L^{\vec{s}})(\mathbb{R}^{n})$ was studied, we can see \cite{LZWO}.
For more meticulous research about mixed-norm spaces, we refer to \cite{AG,LV,ZYZ} and so on.
 
Very recently, Auscher and Mourgoglou \cite{AM} introduced the slice space $\left(E_t^u\right)\left(\mathbb{R}^n\right)$. In 2017, Auscher and Prisuelos-Arribas \cite{AP} studied many classical operators of harmonic analysis on slice space $\left(E_v^u\right)_t\left(\mathbb{R}^n\right)$. In 2022, 
Lu, Zhou, and Wang \cite{LZW} introduced the Herz-slice spaces. Based on these results, we introduce mixed-norm Herz-slice space $(\dot KE_{\vec{u},\vec{v}}^{\beta,s})_t(\rn)$.

Now, we elaborate on the context of this paper.
In sect 2, we introduce the homogeneous mixed-norm Herz-slice space $(\dot KE_{\vec{u},\vec{v}}^{\beta,s})_t(\rn)$ and the non-homogeneous mixed-norm Herz-slice space $(KE_{\vec{u},\vec{v}}^{\beta,s})_t(\rn)$. In sect 3, we establish the dual spaces and study some properties on these spaces. The block decomposition is obtained on mixed-norm Herz-slice spaces in sect 4. In sect 5, we estimate the boundedness of Hardy–Littlewood maximal operator on mixed-norm Herz-slice spaces.

The symbol $\mathscr{K}\left(\mathbb{R}^n\right)$ refers to the set of all measurable functions on $\mathbb{R}^n$. 
We use $\mathbf1_G$ to denote the characteristic function and $|G|$ is the Lebesgue measure of a measurable set $G$. 
Denote $B(y, \lambda)$ the open ball with centered at $y$ with the radius $\lambda$. 
Let $B_m=B(0,2^m)=\{x\in\rn:\,|x|\leq2^m\}$. Denote $S_m:=B_m\setminus B_{m-1}$ for any $m\in\zz$, $\mathbf 1_{m}=\mathbf 1_{S_m}$ for $m\in\zz$, and $\mathbf 1_{S_0}=\mathbf 1_{B_0}$, where $\mathbf 1_{m}$ is the characteristic function of $S_m$. The letters $\vec{u}, \vec{v}, \ldots$ will denote $n$-tuples of the numbers in $[0, \infty], \vec{u}=\left(u_1, \ldots, u_n\right), \vec{v}=$ $\left(v_1, \ldots, v_n\right)$, $n\in \nn$. $0<\vec{u}<\infty$ means that $0<u_i<$ $\infty$ for each $i=1, \cdots, n$. Furthermore, for $\vec{u}=\left(u_1, \ldots, u_n\right)$ and $\eta \in \mathbb{R}$, let
$$
\frac{1}{\vec{u}}=\left(\frac{1}{u_1}, \ldots, \frac{1}{u_n}\right), \quad \frac{\vec{u}}{\eta}=\left(\frac{u_1}{\eta}, \ldots, \frac{u_n}{\eta}\right), \quad \overrightarrow{u^{\prime}}=\left(u_1^{\prime}, \ldots, u_n^{\prime}\right).
$$
Where $u_{i}^{\prime}=\frac{u_{i}}{u_{i}-1}$ is a conjugate exponent of $u_{i}$, $i=1,...,n$. 
For different positive constants we use $C$ to denote them. 
We write $\phi \lesssim \psi$, $\phi \leq C \psi$ mean that  for some constant $C>0$, and $\phi \sim \psi$ means that $\phi \lesssim \psi$ and $\psi \lesssim \phi$.
\section{Main definitions}\label{m}
Let us start by recalling some basic essential notions. 
\begin{definition}({\cite{LY}})
Let $\beta\in\rr$ and $s,\,u\in(0,\infty]$. The homogeneous Herz space
$(\dot K_u^{\beta,s})(\rn)$ is defined by
\begin{equation}\label{DefH}
(\dot K_u^{\beta,s})(\rn):=\lf\{f\in L_{\mathrm{\loc}}^u(\rn\setminus\{0\}):\|f\|_{(\dot K_u^{\beta,s})(\rn)}<\infty\r\},
\end{equation}
where
$$
\|f\|_{(\dot K_u^{\beta,s})(\rn)}:=\lf[\sum_{k=-\infty}^\infty2^{k\beta s}\lf\|f\mathbf1_{S_k}\r\|_{L^u(\rn)}^s\r]^\frac1s,
$$
with the usual modification when $s=\infty$ or $u=\infty$.
And the non-homogeneous Herz space
$(K_u^{\beta,s})(\rn)$ is defined by
\begin{equation}\label{DefnH}
(K_u^{\beta,s})(\rn):=\lf\{f\in L_{\mathrm{loc}}^u(\rn):\|f\|_{(K_u^{\beta,s})(\rn)}<\infty\r\},
\end{equation}
where
$$
\|f\|_{(K_u^{\beta,s})(\rn)}:=\lf[\sum_{k=0}^\infty2^{k\beta s}\lf\|f\mathbf1_{S_k}\r\|_{L^u(\rn)}^s\r]^\frac1s,
$$
with the usual modification when $s=\infty$ or $u=\infty$.
\end{definition}

\begin{definition}({\cite{AP}})
Let $u, v, t \in(0, \infty)$. The slice space $(E_v^u)_t(\rn)$ is defined by
$$
\left(E_v^u\right)_t(\rn):=\left\{f \in L_{\mathrm{\loc}}^1\left(\mathbb{R}^n\right):
\|f\|_{(E_v^u)_t(\rn)}<\infty\right\},
$$
where
$$
\|f\|_{(E_v^u)_t(\rn)}:=\lf\|\lf[\frac1{|B(\cdot,t)|}\int_{B(\cdot,t)}|f(y)|^v\,dy\r]^\frac 1v\r\|_{L^u(\rn)},
$$
with the usual modification when $u=\infty$.
\end{definition}

\begin{definition}({\cite{BP}})
Let $\vec{u}\in(0,\infty]^{n}$. The mixed-norm Lebesgue space $L^{\vec{u}}(\mathbb R^n)$
is defined to be set of all measurable functions $f\in\mathscr{K}(\mathbb R^n)$ such that
$$
\|f\|_{L^{\vec{u}}(\mathbb R^n)}:=\left\{\int_{R}\cdots\left[ \int_{R}|f(x_{1},\ldots,x_{n})|^{u_{1}}dx_{1}\right]^{u_{2}/u_{1}}\cdots dx_{n}\right\}^{1/u_{n}}<\infty.
$$
Obviously, if $u_1=\cdots=u_n=u$, we simplify $L^{\vec{u}}(\mathbb{R}^n)$ to classical Lebesgue spaces $L^u$ and
$$\left\|f\right\|_{L^{u}(\mathbb{R}^n)}=\left[\int_{\mathbb{R}^n}\left|f(x)\right|^{u} dx\right]^{\frac{1}{u}}.$$
When $u_{i}=\infty, i=1,\cdots,n,$ then we make the appropriate modifications.
\end{definition}

\begin{definition}(\cite{ZZM})
Let $t\in(0,\infty)$ and $\vec{v}, \vec{u}\in(1,\fz)^{n}$. The mixed amalgam spaces $(E^{\vec{u}}_{\vec{v}})_t(\mathbb{R}^n)$ is defined as the set of all measurable functions $f$ satisfy $f\in L^{1}_{\mathrm{\loc}}(\mathbb{R}^n)$,
$$(E^{\vec{u}}_{\vec{v}})_t(\mathbb{R}^n) :=\left\{f:\|f\|_{(E^{\vec{u}}_{\vec{v}})_t(\mathbb{R}^n)}
=\left\|\frac{\|f\mathbf1_{B(\cdot,t)}\|_{L^{\vec{v}}(\mathbb{R}^n)}}{\|\mathbf1_{B(\cdot,t)}\|_{L^{\vec{v}}
(\mathbb{R}^n)}}\right\|_{L^{\vec{u}}(\rn)}
<\infty\right\},$$
with the usual modification for $u_i=\infty, i=1,\cdots,n.$
\end{definition}

Now, we give the homogeneous mixed-norm Herz-slice space $(\dot KE_{\vec{u},\vec{v}}^{\beta,s})_t(\rn)$ and the non-homogeneous mixed-norm Herz-slice space $(KE_{\vec{u},\vec{v}}^{\beta,s})_t(\rn)$.
\begin{definition}\label{Def}
Let $\beta\in\rr$, $t\in(0,\infty)$, $s\in(0,\infty]$ and $\vec{v}, \vec{u}\in(1,\fz)^{n}$.

(1)The homogeneous mixed-norm Herz-slice space
$(\dot KE_{\vec{u},\vec{v}}^{\beta,s})_t(\rn)$ is defined by
\begin{equation}\label{Hs}
(\dot KE_{\vec{u},\vec{v}}^{\beta,s})_t(\rn):=\left\{f\in L_{\mathrm{\loc}}^1(\rn) :\lf\|f\r\|_{(\dot KE_{\vec{u},\vec{v}}^{\beta,s})_t(\rn)}<\infty\right\},
\end{equation}
where
$$
\lf\|f\r\|_{(\dot KE_{\vec{u},\vec{v}}^{\beta,s})_t(\rn)}:=\lf[\sum_{k=-\infty}^\infty2^{k\beta s}\lf\|f\mathbf1_{S_k}\r\|_{(E_{\vec{v}}^{\vec{u}})_t(\rn)}^s\r]^\frac1s,
$$
with the usual modification made when $s=\infty$.

(2)The non-homogeneous mixed-norm Herz-slice space $(KE_{\vec{u},\vec{v}}^{\beta,s})_t(\rn)$ is defined by
\begin{equation}
(KE_{\vec{u},\vec{v}}^{\beta,s})_t(\rn):=\left\{f\in L_{\mathrm{\loc}}^1(\rn):\lf\|f\r\|_{( KE_{\vec{u},\vec{v}}^{\beta,s})_t(\rn)}<\infty\right\},
\end{equation}
where
$$
\lf\|f\r\|_{( KE_{\vec{u},\vec{v}}^{\beta,s})_t(\rn)}:=\lf[\sum_{k=0}^\infty2^{k\beta s}\lf\|f\mathbf1_{S_k}\r\|_{(E_{\vec{v}}^{\vec{u}})_t(\rn)}^s\r]^\frac1s,
$$
with the usual modification made when $s=\infty$.
\end{definition}

\begin{remark}
Let $\beta\in\rr$, $t\in(0,\infty)$, $s\in(0,\infty]$ and $\vec{v}, \vec{u}\in(1,\fz)^{n}$.
\begin{itemize}
\item[(1)] $\forall$ $0<s<\infty$, we have
\begin{equation}
\left\||f|^r\right\|_{(\dot KE_{\vec{u}, \vec{v}}^{\beta,s})_t(\rn)}^{1/r}=\|f\|_{(\dot KE_{r{\vec{u}},r{\vec{v}}}^{\beta/r,rs})_t(\rn)}.
\end{equation}
\item[(2)] $(\dot KE_{{\vec{u}}, {\vec{v}}}^{0,s})_t(\rn)=(E_{\vec{v}}^{\vec{u}})_t(\rn)$, when $u_{1}=\cdots=u_{n}, v_{1}=\cdots=v_{n}$,  $(E_{\vec{v}}^{\vec{u}})_t(\rn)=(E^{u}_{v})_t(\rn).$
\item[(3)] Let $\beta\in\rr$, if ${\vec{v}}={\vec{u}}$, $(\dot KE_{\vec{u},\vec{v}}^{\beta,s})_t(\rn)=(\dot K_{\vec{u}}^{\beta,s})(\rn)$ and
$(KE_{\vec{u}, \vec{v}}^{\beta,s})_t(\rn)=(K_{\vec{u}}^{\beta,s})(\rn)$;
  obviously, when $u_{1}=\cdots=u_{n}$, $\dot{K}_{\vec{u}}^{\beta, s}\left(\mathbb{R}^{n}\right)=\dot{K}_{u}^{\beta, s}\left(\mathbb{R}^{n}\right)$,
  ${K}_{\vec{u}}^{\beta, s}\left(\mathbb{R}^{n}\right)={K}_{u}^{\beta, s}\left(\mathbb{R}^{n}\right)$.
\end{itemize}
\end{remark}
In what follows, we give some preliminaries on ball quasi-Banach function spaces introduced in \cite{SHYY}. For any $x \in \mathbb{R}^n$ and $\lambda \in(0, \infty)$, let $B(x, \lambda):=\left\{y \in \mathbb{R}^n:|x-y|<\lambda\right\}$ and
\begin{equation}\label{x}
\mathbb{B}:=\left\{B(x, \lambda): x \in \mathbb{R}^n \:\text{ and } \: \lambda\in(0, \infty)\right\}.
\end{equation}
\begin{definition}\label{defl}
 A quasi-Banach space $X \subset \mathscr{K}\left(\mathbb{R}^n\right)$ is called a ball quasi-Banach function space if it satisfies

(1) $\|\phi\|_X=0$ implies that $\phi=0$ almost everywhere;

(2) $\psi| \leq|\phi|$ almost everywhere implies that $\|\psi\|_X \leq\|\phi\|_X$;

(3) $0 \leq \phi_m \uparrow \phi$ almost everywhere implies that $\left\|\phi_m\right\|_X \uparrow\|\phi\|_X$;

(4) $B \in \mathbb{B}$ implies that $\mathbf{1}_B \in X$, where $\mathbb{B}$ is as in (\ref{x}).

Moreover, a ball quasi-Banach function space $X$ is called a ball Banach function space if the norm of $X$ satisfies the triangle inequality:

(5) for any $\phi, \psi \in X$,
\begin{equation}
\|\phi+\psi\|_X \leq\|\phi\|_X+\|\psi\|_X,
\end{equation}
and, for any $B \in \mathbb{B}$, there exists a positive constant $C_{(B)}$, depending on $B$, such that,

(6) for any $\phi \in X$
\begin{equation}
\int_B|\phi(x)| d x \leq C_{(B)}\|\phi\|_X .
\end{equation}
\end{definition}

Let us recall the notion of the Hardy-Littlewood maximal operator $\mathscr{M}$.
\begin{definition}
For any $\phi \in$ $L_{\mathrm{loc} }^1\left(\mathbb{R}^n\right)$ and $x \in \mathbb{R}^n$, we define the Hardy-Littlewood maximal function  $\mathscr{M}(\phi)$ by
\begin{equation}
\mathscr{M}(\phi)(x):=\sup _B \frac{1}{|B|} \int_B|\phi(y)| d y,
\end{equation}
where the supremum is taken over all balls $B \in \mathbb{B}$ containing $x$.
\end{definition}

\section{Properties of mixed-norm Herz-slice spaces}\label{H}
In this section, we first present the dual spaces of mixed-norm Herz-slice spaces, then show some elementary properties on mixed-norm Herz-slice spaces. Before statement the dual spaces of mixed-norm Herz-slice spaces, we first recall H\"older's inequality on $(E^{\vec{u}}_{\vec{v}})_t(\rn)$.  
\begin{lemma}\label{l}(\cite{ZZM}) 
Let $t\in(0,\infty)$ and $\vec{v}, \vec{u}\in(1,\fz)^{n}$. If $\phi\in(E^{\vec{u}}_{\vec{v}})_t(\rn)$ and $\psi\in(E^{{\vec{u'}}}_{{\vec{v'}}})_t(\rn)$, then $\phi\psi$ is integrable and
$$
\|\phi\psi\|_{L^1(\mathbb R^n)}\leq\|\phi\|_{(E^{\vec{u}}_{\vec{v}})_t(\rn)}\|\psi\|_{(E^{{\vec{u\prime}}}_{{\vec{v\prime}}})_t(\rn)},
$$
where $1/\vec{u}+1/\vec{u'}=1/\vec{v}+1/\vec{v'}=1$.
\end{lemma}
We began to prove the dual space of $(\dot KE_{\vec{u}, \vec{v}}^{\beta,s})_t(\rn)$.
\begin{theorem}\label{KDuiO}
Let $\beta\in\rr$, $t,s\in(0,\infty)$ and $\vec{v}, \vec{u}\in(1,\fz)^{n}$. The dual space of
$(\dot KE_{\vec{u},\vec{v}}^{\beta,s})_t(\rn)$ is
$$
\left(\left(\dot KE_{\vec{u},\vec{v}}^{\beta,s}\right)_t(\rn)\right)^*=\begin{cases}
\left(\dot KE_{{\vec{u\prime}},{\vec{v\prime}}}^{-\beta,s\prime}\right)_t(\rn),\quad 1<s<\infty,\\
\left(\dot KE_{{\vec{u\prime}},{\vec{v\prime}}}^{-\beta,\infty}\right)_t(\rn),\quad 0<s\leq1.
\end{cases}
$$
\end{theorem}
\begin{proof}
Let $1<s<\infty$ and $L\in(\dot KE_{{\vec{u\prime}},{\vec{v\prime}}}^{-\beta,s\prime})_t(\rn)$.
For any $g\in(\dot KE_{\vec{u},\vec{v}}^{\beta,s})_t(\rn)$, we define
$$
(L,g):=\int_\rn L(x)g(x)\,dx=\sum_{l=-\infty}^\infty\int_{S_l}L(x)g(x)\,dx.
$$
By Lemma \ref{l}, we have
\begin{align*}
|(L,g)|
&\leq\left[\sum_{l=-\infty}^\infty2^{-l\beta s'}\|L\|_{(E^{{\vec{u\prime}}}_{{\vec{v'}}})_t(S_l)}^{s\prime}\right]^\frac1{s\prime}
\left[\sum_{l=-\infty}^\infty2^{l\beta s}\|g\|_{(E^{\vec{u}}_{\vec{v}})_t(S_l)}^s\right]^\frac1{s}\\
&=\|L\|_{(\dot KE_{{\vec{u\prime}},{\vec{v\prime}}}^{-\beta,s\prime})_t(\rn)}\|g\|_{(\dot KE_{\vec{u},\vec{v}}^{\beta,s})_t(\rn)}.
\end{align*}
This indicates that $L\in((\dot KE_{\vec{u},\vec{v}}^{\beta,s})_t(\rn))^*$.

Let $L\in ((\dot KE_{\vec{u},\vec{v}}^{\beta,s})_t(\rn))^*$. For any $l\in\zz$ and $g_l\in (E_{\vec{v}}^{\vec{u}})_t(\rn)$, write
$\widetilde{g_l}:= g_l\mathbf1_{S_l}$, then $\widetilde g_l\in (\dot KE_{\vec{u},\vec{v}}^{\beta,s})_t(\rn)$,
$\lf\|g_{l}\r\|_{(\dot KE_{\vec{u},\vec{v}}^{\beta,w})_t(\rn)}=2^{l\alpha}\|\widetilde g_l\|_{(E_{\vec{v}}^{\vec{u}})_t(\rn)}$.

Let $(L_l,g_l):=(L,\widetilde g_l)$. We could easily know that $L_l\in (E_{{\vec{v\prime}}}^{{\vec{u\prime}}})_t(\rn)$ and $$\lf\|L_l\r\|_{(E_{{\vec{v\prime}}}^{{\vec{u\prime}}})_t(\rn)}\leq2^{l\beta}\|L\|_{(\dot KE_{\vec{u},\vec{v}}^{\beta,s})_t(\rn))^*}.$$
Additionally, for any $g\in(\dot KE_{\vec{u},\vec{v}}^{\beta,s})_t(\rn)$, we have $g\mathbf1_{S_l}\in (E_{\vec{v}}^{\vec{u}})_t(\rn)$.
Then, when given $P, Q\in\mathbb N$,
$$
\sum_{l=-P}^Q(L_l,g\mathbf1_{S_l})=\left(L,\sum_{l=-P}^Qg\mathbf1_{S_l}\right).
$$
For any $l\in\zz$, take $g_l'\in(E_{\vec{v}}^{\vec{u}})_t(\rn)$ with $\supp (g_l^{\prime})\subset S_l$ and $\lf\|f_l'\r\|_{(E_{\vec{v}}^{\vec{u}})_t(\rn)}=2^{-l\beta}$ such that
$$
\left(L_l,g_l'\r)\geq 2^{-l\beta}\lf\|L_l\r\|_{(E_{{\vec{v\prime}}}^{{\vec{u\prime}}})_t(\rn)}-\delta_l,
$$
where $\delta_l>0$ is decided subsequently. Let
$$
g_l:=\left(2^{-l\beta}\|L_l\|_{(E_{{\vec{v\prime}}}^{{\vec{u\prime}}})_t(\rn)}\right)^{s\prime-1}g_l'.
$$
For any given $\delta\in(0,\infty)$ , choose $\delta_l>0$ small enough such that
$$
(L_l,g_l)+ 2^{-|l|}\delta\geq2^{-l\beta s'}\|L_l\|_{(E_{{\vec{v\prime}}}^{{\vec{u\prime}}})_t(\rn)}^{s\prime}=2^{l\beta s}\lf\|g_l\r\|_{(E_{\vec{v}}^{\vec{u}})_t(\rn)}^s.
$$
Then we easily get
\begin{align*}
&\sum_{l=-P}^Q2^{-l\beta s\prime}\lf\|L_k\mathbf1_{S_l}\r\|_{(E_{{\vec{v\prime}}}^{{\vec{u\prime}}})_t(\rn)}^{s\prime}\\
&\leq 4\varepsilon+\left(L,\sum_{l=-P}^Qg_l\mathbf1_{S_l}\right)\\
&\leq 4\varepsilon+\|L\|_{((\dot KE_{\vec{u},\vec{v}}^{\beta,s})_t(\rn))^*}\left\|\sum_{l=-P}^Qg_l\right\|_{(\dot KE_{\vec{u},\vec{v}}^{\beta,s})_t(\rn)}\\
&\leq 4\varepsilon+\|L\|_{((\dot KE_{\vec{u},\vec{v}}^{\beta,s})_t(\rn))^*}\lf(\sum_{l=-P}^Q2^{l\beta p}\lf\|g_l\mathbf1_{S_l}\r\|_{(E_{\vec{v}}^{\vec{u}})_t(\rn)}^s\r)^\frac1s\\
&=4\varepsilon+\|L\|_{((\dot KE_{\vec{u},\vec{v}}^{\beta,s})_t(\rn))^*}\lf(\sum_{l=-P}^Q2^{-l\beta s'}\lf\|L_l\mathbf1_{S_l}\r\|_{(E_{{\vec{v\prime}}}^{{\vec{u\prime}}})_t(\rn)}^{s\prime}\r)^\frac1s.
\end{align*}
Using $\delta\rightarrow0$ and $P, Q\rightarrow\infty$, we conclude that
\begin{equation}\label{Eqri}
\left[\sum_{l=-\infty}^\infty2^{-l\beta s'}\lf\|L_l\mathbf1_{S_l}\r\|_{(E_{{\vec{v\prime}}}^{{\vec{u\prime}}})_t(\rn)}^{s\prime}\right]^\frac1{s\prime}\leq \lf\|L\r\|_{((\dot KE_{\vec{u},\vec{v}}^{\beta,s})_t(\rn))^*}.
\end{equation}
Define
$$
\widetilde L(x):=\sum_{l=-\infty}^\infty L_l(x)\mathbf1_{S_l}(x).
$$
From \eqref{Eqri} we know that $\widetilde T\in (\dot KE_{{\vec{u\prime}},{\vec{v\prime}}}^{-\beta,s\prime})_t(\rn)$. Furthermore, for any $g\in(\dot KE_{\vec{u},\vec{v}}^{\beta,s})_t(\rn)$,
we can see that
\begin{align*}
\lf(\widetilde L,g\r)
&=\sum_{l=-\infty}^\infty L_l(x)\mathbf1_{S_l}(x)g(x)dx,\\
&=\sum_{l=-\infty}^\infty\int_{S_l}L_l(x)g(x)\,dx=\sum_{l=-\infty}^\infty(L,g\mathbf1_{S_l})=(L,g).
\end{align*}
Thus, $L$ is also in $(\dot KE_{\vec{u\prime},\vec{v\prime}}^{-\beta,s\prime})_t(\rn)$. So for $1<{s}<\infty$,
we will omit the details since the proof is similar.
\end{proof}
Based on the closed-graph theorem, we immediately receive the following result.
\begin{corollary}\label{NN}
Let $\beta\in\rr$, $t,s\in(0,\infty)$ and $\vec{v}, \vec{u}\in(1,\fz)^{n}$ with $1/s+1/s\prime=1/{\vec{v}}+1/{\vec{v\prime}}=1/{\vec{u}}+1/{\vec{u\prime}}=1$.
Then $\phi\in(\dot KE_{\vec{u},\vec{v}}^{\beta,s})_t(\rn)$ if and only if
$$
\int_\rn \phi(x)\psi(x)\,dx<\fz,
$$
with $\psi\in (\dot KE_{\vec{u\prime},\vec{v\prime}}^{-\beta,s\prime})_t(\rn)$, and
$$
\lf\|\phi\r\|_{(\dot KE_{\vec{u},\vec{v}}^{\beta,s})_t(\rn)}=
\sup \lf\{\int_\rn \phi(x)\psi(x)\,dx:\lf\|\psi\r\|_{(\dot KE_{{\vec{u\prime}},\vec{v\prime}}^{-\beta,s\prime})_t(\rn)}\leq 1\r\}.
$$
\end{corollary}
The next assertion follows in the same way as the previous one.
\begin{theorem}\label{oKDuiO}
Let $\beta\in\rr$, $t,s\in(0,\infty)$ and $\vec{v}, \vec{u}\in(1,\fz)^{n}$. The dual space of
$(KE_{\vec{u},\vec{v}}^{\beta,s})_t(\rn)$ is
$$
\left((KE_{\vec{u},\vec{v}}^{\beta,s})_t(\rn)\right)^*=\begin{cases}
\left(KE_{{\vec{u\prime}},{\vec{v\prime}}}^{-\beta,s'}\right)_t(\rn),\quad 1<s<\infty,\\
\left(KE_{{\vec{u\prime}},{\vec{v\prime}}}^{-\beta,\infty}\right)_t(\rn),\quad 0<s\leq1.
\end{cases}
$$
In addition, assume that $\beta\in\rr$, $t,s\in(0,\infty)$, $\vec{v}, \vec{u}\in(1,\infty)^{n}$ and $1/s+1/s\prime=1,1/\vec{v}+1/{\vec{v\prime}}=1/\vec{u}+1/{\vec{u\prime}}=1$,
then $\phi\in(KE_{u,v}^{\beta,s})_t(\rn)$ if and only if
$$
\int_\rn \phi(x)\psi(x)\,dx<\fz,
$$
where $\psi\in (KE_{\vec{u\prime},\vec{v\prime}}^{-\beta,s\prime})_t(\rn)$, and
$$
\lf\|\phi\r\|_{(KE_{\vec{u},\vec{v}}^{\beta,s})_t(\rn)}=
\sup \lf\{\int_\rn \phi(x)\psi(x)\,dx:\lf\|\psi\r\|_{(KE_{{\vec{u\prime}},{\vec{v\prime}}}^{-\beta,s\prime})_t(\rn)}\leq 1\r\}.
$$
\end{theorem}

Theorem 3.1 further implies that the following lemma, we acknowledge the details.
\begin{corollary}\label{ll}
Let $\beta\in\rr$, $t,s\in(0,\infty)$ and $\vec{v}, \vec{u}\in(1,\fz)^{n}$. If $\phi\in(\dot KE_{\vec{u},\vec{v}}^{\beta,s})_t(\rn)$ and $\psi\in (\dot KE_{{\vec{u\prime}},{\vec{v\prime}}}^{-\beta,s\prime})_t(\rn)$, then $\phi\psi$ is integrable and
$$
\|\phi \psi\|_{L^1(\mathbb R^n)}\leq\|\phi\|_{(\dot KE_{\vec{u},\vec{v}}^{\beta,s})_t(\rn)}\|\psi\|_{(\dot KE_{{\vec{u\prime}},{\vec{v\prime}}}^{-\beta,s\prime})_t(\rn)}.
$$
\end{corollary}

Before we get into the properties of the investigation, let's make an important conclusion.
\begin{lemma}\label{lll}
Let $t\in(0,\infty)$ and $\vec{v}, \vec{u}\in(1,\fz)^{n}$. The characteristic function
on $B(y_0,\lambda)$ with $y_0\in\rn$ and $1<\lambda<\fz$ satisfies
\begin{equation}
\lf\|\mathbf1_{B(y_0,\lambda)}\r\|_{(E_{\vec{v}}^{\vec{u}})_t(\rn)}\lesssim \lambda^{\sum_{i=1}^n\frac{1}{u_i}}.
\end{equation}
\end{lemma}
\begin{proof}
If $t> R_{0}$, then
\begin{align*}
 \left\|\mathbf1_{B\left(y_0, \lambda\right)}\right\|_{(E_{\vec{v}}^{\vec{u}})_t(\rn)}
&\leq\left\|\frac{\| \mathbf1_{B\left(y_0, \lambda\right)} \mathbf1_{B(\cdot, t)}\|_{L^{\vec{v}}}}
{\| \mathbf1_{B(\cdot, t)}\|_{L^{\vec{v}}}}\right\|_{L^{\vec{u}}}
\leq\left\|\frac{\| \mathbf1_{B\left(y_0, \lambda+t\right)}\|_{L^{\vec{v}}}\mathbf1_{B\left(y_0, \lambda\right)}}
{\| \mathbf1_{B(\cdot, t)}\|_{L^{\vec{v}}}}\right\|_{L^{\vec{u}}}\\
&\leq\frac{\left(\lambda+t\right)^{\sum_{i=1}^n \frac{1}{v_{i}}}}{t^{\sum_{i=1}^n \frac{1}{v_i}}}\left\|\mathbf1_{B\left(y_0, \lambda\right)}\right\|_{L^{\vec{u}}\leq}
\leq C\lambda^{\sum_{i=1}^n \frac{1}{u_i}},
\end{align*}
if $t\leq R_{0}$, then
$$
\begin{aligned}
 \left\|\mathbf1_{B\left(y_0, \lambda\right)}\right\|_{(E_{\vec{v}}^{\vec{u}})_t(\rn)}
&\leq\left\|\frac{\| \mathbf1_{B\left(y_0, \lambda\right)} \mathbf1_{B(\cdot, t)}\|_{L^{\vec{v}}}}
{\|\mathbf1_{B(\cdot, t)}\|_{L^{\vec{v}}}}\right\|_{L^{\vec{u}}}
\leq\left\|\frac{\| \mathbf1_{B\left(\cdot,t\right)}\|_{L^{\vec{v}}}\mathbf1_{B\left(y_0, \lambda+t\right)}}
{\|\mathbf1_{B(\cdot, t)}\|_{L^{\vec{v}}}}\right\|_{L^{\vec{u}}}\\
& \leq\frac{t^{\sum_{i=1}^n \frac{1}{v_i}} \cdot\left(2\lambda\right)^{\sum_{i=1}^n \frac{1}{u_{i}}}}{t^{\sum_{i=1}^n \frac{1}{v_i}}}
\leq C\lambda^{\sum_{i=1}^n \frac{1}{u_i}}.
\end{aligned}
$$
This accomplishes the desired result.
\end{proof}

\begin{remark}\label{r}
Let $t\in(0,\infty)$ and $\vec{v}, \vec{u}\in(1,\fz)^{n}$. A characteristic function
on $S_k$ satisfies
\begin{equation}
\lf\|\mathbf1_{S_k}\r\|_{(E_{\vec{v}}^{\vec{u}})_t(\rn)}\leq \lf\|\mathbf1_{B_k}\r\|_{(E_{\vec{v}}^{\vec{u}})_t(\rn)}\lesssim 2^{k\sum_{i=1}^n\frac{1}{u_i}}.
\end{equation}
\end{remark}

\begin{proposition}\label{pl}
Let $t\in(0,\infty)$ and $\vec{v}, \vec{u}\in(1,\fz)^{n}$. The mixed-norm slice spaces $(E^{\vec{u}}_{\vec{v}})_t(\rn) $ is a ball Banach function space.
\end{proposition}
Before the proof of Proposition \ref{pl}, we need some preliminary lemmas.

\begin{lemma}\label{llll}
 Let $t\in(0,\infty)$ and $\vec{v}$, $\vec{u}\in(1,\fz)^{n}$. Let $\phi, \psi \in \mathscr{K}\left(\mathbb{R}^n\right)$. Assume that $|\phi| \leq|\psi|$ almost everywhere on $\mathbb{R}^n$, then
$$
\|\phi\|_{(E^{\vec{u}}_{\vec{v}})_t(\rn)} \leq\|\psi\|_{(E^{\vec{u}}_{\vec{v}})_t(\rn)}.
$$
\end{lemma}
\begin{proof}
Let all the signs be like this proposition. Let
$$
G:=\left\{x \in \mathbb{R}^n:\left|\psi\left(x\right)\right|<\left|\phi\left(x\right)\right|\right\}.
$$
Where $x=(x_{1},\cdots, x_{n})$.
Suppose that $|\phi| \leq|\psi|$ almost everywhere on $\mathbb{R}^n$, then $|G|=0$.
Thus, for almost every $x \in \mathbb{R}^{n}, \left|\phi\right| \leq\left|\psi\right|$ almost everywhere on $\mathbb{R}^{n}$, by \cite[Subsection 4.2]{TYY}, if $|\phi|\leq|\psi|$ almost everywhere on $\mathbb{R}^n$, then 
$$
\left\|\phi\left(x\right) \mathbf1_{B(\cdot,t)}\right\|_{L^{\vec{v}}(\mathbb{R}^{n})} \leq\left\|\psi\left(x\right)\mathbf1_{B(\cdot,t)}\right\|_{L^{\vec{v}}(\mathbb{R}^{n})},
$$
namely
$$
\frac{\left\|\phi\left(x\right) \mathbf1_{B(\cdot,t)}\right\|_{L^{\vec{v}}(\mathbb{R}^{n})}}{\left\| \mathbf1_{B(\cdot,t)}\right\|_{L^{\vec{v}}(\mathbb{R}^{n})}} \leq \frac{\left\|\psi\left(x\right) \mathbf1_{B(\cdot,t)}\right\|_{L^{\vec{v}}(\mathbb{R}^{n})}}{\left\| \mathbf1_{B(\cdot,t)}\right\|_{L^{\vec{v}}(\mathbb{R}^{n})}}.
$$
Using Definition 2.2, for almost every $x \in \mathbb{R}^{n}$ and any given $\vec{u}\in(1,\fz)^{n}$, we easily obtain,
$$
\left\|\phi\right\|_{(E^{\vec{u}}_{\vec{v}})_t(\rn)} \leq\left\|\psi\right\|_{(E^{\vec{u}}_{\vec{v}})_t(\rn)}.
$$
This accomplishes the desired result.
\end{proof}

\begin{lemma}\label{Ll}
Let $0<t<\infty$ and $\vec{v}, \vec{u}\in(1,\fz)^{n}$. For any $\phi \in (E^{\vec{u}}_{\vec{v}})_t(\rn)$, assume that $\|\phi\|_{(E^{\vec{u}}_{\vec{v}})_t(\rn)}=0$, then $\phi=0$ almost everywhere.
\end{lemma}

\begin{proof}
From Remark 2.1 and Lemma \ref{llll}, for any $\phi \in (E^{\vec{u}}_{\vec{v}})_t(\rn)$, $r\in(0,\infty)$ and $m \in \mathbb{N}$, we have 

$$
\left\||\phi|^{r} \mathbf{1}_{S_m}\right\|_{(E^{\vec{u}/r}_{\vec{v}/r})_t(\rn)}^{1/r}=\left\|\phi \mathbf{1}_{S_m}\right\|_{(E^{\vec{u}}_{\vec{v}})_t(\rn)} \leq\|\phi\|_{(E^{\vec{u}}_{\vec{v}})_t(\rn)}=0.
$$
Using Lemma \ref{l} and \ref{lll}, it suffices to prove that
$$
\left\||\phi|^r \mathbf{1}_{S_m}\right\|_{L^{1}(\rn)}
\leq\left\||\phi|^r \mathbf{1}_{S_m}\right\|_{(E^{\vec{u}/r}_{\vec{v}/r})_t(\rn)}
\left\||\phi|^r \mathbf{1}_{S_m}\right\|_{(E^{\vec{u'}/r}_{\vec{v'}/r})_t(\rn)}
\leq 0.
$$
 We observe that $\left\| |\phi|^r\right\|_{L^1\left(\mathbb{R}^n\right)}=0$ via the monotone convergence theorem, it follows from this, $\phi=0$ almost everywhere on $\mathbb{R}^n$. The expected results were obtained.
\end{proof}

\begin{lemma}\label{Lll}
Let $0<t<\infty$ and $\vec{v}, \vec{u}\in(1,\fz)^{n}$. Assume that $\left\{\phi_k\right\}_{k \in \mathbb{N}}$ is measurable functions and $\phi_k \geq 0$ with $k \in \mathbb{N}$, then
$$
\left\|\lim _{k \rightarrow \infty} \phi_k\right\|_{(E^{\vec{u}}_{\vec{v}})_t(\rn)}
 \leq \lim _{k \rightarrow \infty}\left\|\phi_k\right\|_{(E^{\vec{u}}_{\vec{v}})_t(\rn)} .
$$
\end{lemma}

\begin{proof}
Let $0<t<\infty$ and $\vec{v}, \vec{u}\in(1,\fz)^{n}$.
We know that $L^{\vec{u}}(\rn)$ is ball Banach function spaces via \cite[Subsection 4.2]{TYY}. Then, we deduce that, when given $\vec{v}\in(1,\fz)^{n}$
$$
\left\|\varliminf_{k \rightarrow \infty} \phi_k \mathbf{1}_{Q(\cdot,t)}\right\|_{L^{\vec{v}}(\mathbb{R}^{n})}
\leq \lim _{k \rightarrow \infty}\left\|\phi_k \mathbf{1}_{Q(\cdot,t)}\right\|_{L^{\vec{v}}(\mathbb{R}^{n})} .
$$
Namely
$$
\frac{\left\|\varliminf_{k \rightarrow \infty} \phi_k \mathbf{1}_{Q(\cdot,t)}\right\|_{L^{\vec{v}}(\mathbb{R}^{n})} }{\left\|\mathbf{1}_{Q(\cdot,t)}\right\|_{L^{\vec{v}}(\mathbb{R}^{n})} }
\leq
\frac{\lim _{k \rightarrow \infty}\left\|\phi_k \mathbf{1}_{Q(\cdot,t)}\right\|_{L^{\vec{v}}(\mathbb{R}^{n})}}
{\left\|\mathbf{1}_{Q(\cdot,t)}\right\|_{L^{\vec{v}}(\mathbb{R}^{n})}}.
$$
Using Definition 2.2 we easily obtain, for any given $\vec{u}\in(1,\fz)^{n}$,
\begin{align*}
\left\|\frac{\left\|\varliminf_{k \rightarrow \infty} \phi_k \mathbf{1}_{Q(\cdot,t)}\right\|_{L^{\vec{v}}(\mathbb{R}^{n})} }{\left\|\mathbf{1}_{Q(\cdot,t)}\right\|_{L^{\vec{v}}(\mathbb{R}^{n})} } \right\|_{L^{\vec{u}}(\mathbb{R}^{n})}
&\leq
\left\|\frac{\lim _{k \rightarrow \infty}\left\|\phi_k \mathbf{1}_{Q(\cdot,t)}\right\|_{L^{\vec{v}}(\mathbb{R}^{n})}}
{\left\|\mathbf{1}_{Q(\cdot,t)}\right\|_{L^{\vec{v}}(\mathbb{R}^{n})}} \right\|_{L^{\vec{u}}(\mathbb{R}^{n})}\\
&=
\lim _{k \rightarrow \infty}\left\|\frac{\left\|\phi_k \mathbf{1}_{Q(\cdot,t)}\right\|_{L^{\vec{v}}(\mathbb{R}^{n})}}
{\left\|\mathbf{1}_{Q(\cdot,t)}\right\|_{L^{\vec{v}}(\mathbb{R}^{n})}} \right\|_{L^{\vec{u}}(\mathbb{R}^{n})}.
\end{align*}
The expected results were obtained.
\end{proof}

Combining Lemma \ref{Ll} and \ref{Lll}, we can get the following result.
\begin{corollary}\label{Llll}
Let $t \in (0, \infty)$ and $\vec{v}, \vec{u}\in(1,\fz)^{n}$. Assume that $0\leq \phi_m\uparrow \phi$ almost everywhere as $m \rightarrow \infty$, then $\left\|\phi_m\right\|_{(E^{\vec{u}}_{\vec{v}})_t(\rn)} \uparrow\|\phi\|_{(E^{\vec{u}}_{\vec{v}})_t(\rn)}$ as $m \rightarrow \infty$.
\end{corollary}

\begin{proof}[Proof of Proposition \ref{pl}]
 Let $t\in(0,\infty)$ and $\vec{v}, \vec{u}\in(1,\fz)^{n}$. By Lemma \ref{Ll}, \ref{llll}, \ref{lll}, and Corollary \ref{Llll}, we know that the space $(E^{\vec{u}}_{\vec{v}})_t(\rn)$ satisfies (1), (2), (3) and (4) of Definition \ref{defl}. Proof of triangle inequality is analogue to that of $L^{\vec{u}}(\rn)$. So, only need to check (5) of Definition \ref{defl}.
 by Lemma \ref{l}, \ref{llll} and Remark \ref{r}. Notice that, for any $B(z, r)\in\mathbb{B}$, we deduce that
$$
\left|\int_{\mathbb{R}^n} f(x)\mathbf{1}_B(x) d x\right| \leq \left\|f\mathbf{1}_B\right\|_{(E^{\vec{u}}_{\vec{v}})_t(\rn)}\left\| \mathbf{1}_B\right\|_{(E^{\vec{u'}}_{\vec{v'}})_t(\rn)}
\leq C\left\| f\right\|_{(E^{\vec{u}}_{\vec{v}})_t(\rn)}.
$$
The expected results were obtained.
\end{proof}

\begin{proposition}\label{pll}
Let $\beta\in\rr$, $t,s\in(0,\fz)$ and $\vec{v}$, $\vec{u} \in(1,\fz)^{n}$. $(\dot KE_{\vec{u},\vec{v}}^{\beta,s})_t(\rn)$ and $(KE_{\vec{u},\vec{v}}^{\beta,s})_t(\rn)$ is a ball quasi-Banach function space if and only if $\beta\in\left(-\sum_{i=1}^n{1}/{u_i}, \infty\right)$ with $i \in\{1, \ldots, n\}$.

\end{proposition}
\begin{remark}
As a special case, Wei and Yan obtained  Proposition \ref{pl} and Proposition \ref{pll} in
\cite{WY}, it is pointed out here we get  Proposition \ref{pl} and Proposition \ref{pll}  via a direct way rather than properties of ball Banach function spaces in \cite{WY}.
\end{remark}
To prove Proposition \ref{pll}, we need some assertions.

\begin{lemma}\label{Lllll}
Let $\beta\in\rr$, $t,s\in(0,\fz)$ and $\vec{v}, \vec{u} \in(1,\fz)^{n}$. Let $\phi, \psi\in \mathscr{K}\left(\mathbb{R}^n\right)$. Assume that $|\phi| \leq|\psi|$ almost everywhere on $\mathbb{R}^n$, then
$$
\|\phi\|_{(\dot KE_{\vec{u},\vec{v}}^{\beta,s})_t(\rn)} \leq\|\psi\|_{(\dot KE_{\vec{u},\vec{v}}^{\beta,s})_t(\rn)},
$$
and
$$
\|\phi\|_{( KE_{\vec{u},\vec{v}}^{\beta,s})_t(\rn)} \leq\|\psi\|_{( KE_{\vec{u},\vec{v}}^{\beta,s})_t(\rn)}.
$$
\end{lemma}
\begin{proof}
We are just proof the result of ${(\dot KE_{\vec{u},\vec{v}}^{\beta,s})_t(\rn)}$. Let
$$
G:=\left\{x \in \mathbb{R}^n:\left|\psi\left(x\right)\right|<\left|\phi\left(x\right)\right|\right\}.
$$
Where $x=(x_{1},\cdots, x_{n})$.
Assume that $|\phi(x)| \leq|\psi(x)|$ almost everywhere on $\mathbb{R}^n$, then $|G|=0$.
Therefore, $\left|\phi\left(x\right)\right| \leq\left|\psi\left(x\right)\right|$ almost everywhere on $\mathbb{R}^{n}$.
For almost every $x\in \mathbb{R}^{n}$, by Lemma \ref{llll} and (\ref{Hs}), we have
$$
\left\|\phi\right\|_{(\dot KE_{\vec{u},\vec{v}}^{\beta,s})_t(\rn)} \leq\left\|\psi\right\|_{(\dot KE_{\vec{u},\vec{v}}^{\beta,s})_t(\rn)}.
$$
This completes the result of the proof.
\end{proof}
By Remark \ref{r}, we show that the following Corollary, we acknowledge the details.
\begin{corollary}\label{lllll}
 Let $\beta\in\rr$, $t,s\in(0,\fz)$ and $\vec{v}, \vec{u} \in(1,\fz)^{n}$. If $\beta\in\left(-\sum_{i=1}^n{1}/{u_i}, \infty\right)$ with $i \in\{1, \ldots, n\}$. Then, for any $m\in \mathbb{N}$, we have $\mathbf{1}_{S_m} \in (\dot KE_{\vec{u},\vec{v}}^{\beta,s})_t(\rn)$ and $\mathbf{1}_{S_m} \in (KE_{\vec{u},\vec{v}}^{\beta,s})_t(\rn)$.
\end{corollary}

\begin{lemma}\label{LL}
 Let $\beta\in\rr$, $t,s\in(0,\fz)$ and $\vec{v}, \vec{u} \in(1,\fz)^{n}$. For any $f \in (\dot KE_{\vec{u},\vec{v}}^{\beta,s})_t(\rn)$, if $\|f\|_{(\dot KE_{\vec{u},\vec{v}}^{\beta,s})_t(\rn)}=0$, then $f=0$ almost everywhere.
\end{lemma}

\begin{proof}
From (1) of Remark 2.1 and Lemma \ref{lllll}, we then see that
for any $f \in (E^{\vec{u}}_{\vec{v}})_t(\rn)$, $s\in(0,\infty)$ and $m \in \mathbb{N}$,
$$
\left\||f|^r \mathbf{1}_{S_m}\right\|_{(\dot KE_{r{\vec{u}}/r,{\vec{v}}/r}^{r\beta,s/r})_t(\rn)}^{1 / r}=\left\|f \mathbf{1}_{S_m}\right\|_{(\dot KE_{\vec{u},\vec{v}}^{\beta,s})_t(\rn)} \leq\|f\|_{(\dot KE_{\vec{u},\vec{v}}^{\beta,s})_t(\rn)}=0,
$$
where $S_m$ is as in Lemma \ref{lllll}, which, together with Lemma \ref{ll} and Corollary \ref{NN}, we have
$$
\left\||f|^r \mathbf{1}_{S_m}\right\|_{L^{1}(\rn)}
\leq\left\||f|^r \mathbf{1}_{S_m}\right\|_{(\dot KE_{{\vec{u}}/r,{\vec{v}/r}}^{r\beta,s/r})_t(\rn)}
\left\|\mathbf{1}_{S_m}\right\|_{(\dot KE_{{\vec{u'}/r},{\vec{v'}/r}}^{-r\beta,s'/r})_t(\rn)}
\leq 0.
$$
Thus $f=0$ almost everywhere on $\mathbb{R}^n$, because we can see that $\left\| |f|^r\right\|_{L^1\left(\mathbb{R}^n\right)}=0$ via the monotone convergence theorem. This accomplishes the desired result.
\end{proof}

\begin{lemma}\label{LLl}
 Let $\beta\in\rr$, $t,s\in(0,\fz)$ and $\vec{v}, \vec{u} \in(1,\fz)^{n}$. Assume that $\left\{\phi_\tau\right\}_{\tau \in \mathbb{N}}$ is measurable functions and $\phi_\tau \geq 0$ with $\tau \in \mathbb{N}$, then
$$
\left\|\lim _{\tau \rightarrow \infty} \phi_\tau\right\|_{(\dot KE_{\vec{u},\vec{v}}^{\beta,s})_t(\rn)} \leq \lim _{\tau \rightarrow \infty}\left\|\phi_\tau\right\|_{(\dot KE_{\vec{u},\vec{v}}^{\beta,s})_t(\rn)}.
$$
\end{lemma}

\begin{proof}
By Lemma \ref{Lll}, if $s\in(0, \infty)$, using Fatou lemma and (\ref{Hs}), we have
$$
\left\|\lim _{\tau \rightarrow \infty}\phi_\tau\right\|_{(\dot KE_{\vec{u},\vec{v}}^{\beta,s})_t(\rn)}
 \leq\lf[\sum_{\tau=-\infty}^\infty2^{k\beta s}\lim _{\tau \rightarrow \infty} \lf\|\phi\mathbf1_{S_\tau}\r\|_{(E_{\vec{v}}^{\vec{u}})_t(\rn)}^s\r]^\frac1s
 \leq \lim _{\tau \rightarrow \infty}\left\|\phi_\tau\right\|_{(\dot KE_{\vec{u},\vec{v}}^{\beta,s})_t(\rn)}.
$$
This accomplishes the desired result.
\end{proof}

Lemma \ref{LL} and \ref{LLl} imply the following assertion.
\begin{corollary}\label{LLll}
 Let $\beta\in\rr$, $t,s\in(0,\fz)$ and $\vec{v}, \vec{u} \in(1,\fz)^{n}$. Let $\phi \in \mathscr{K}\left(\mathbb{R}^n\right)$ and $\left\{\phi_\tau\right\}_{\tau \in \mathbb{N}} \subset \mathscr{K}\left(\mathbb{R}^n\right)$. 
 Assume that $0 \leq \phi_\tau \uparrow \phi$ almost everywhere when $\tau \rightarrow \infty$, then $\left\|\phi_\tau\right\|_{(\dot KE_{\vec{u},\vec{v}}^{\beta,s})_t(\rn)} \uparrow\|\phi\|_{(\dot KE_{\vec{u},\vec{v}}^{\beta,s})_t(\rn)}$ when $\tau \rightarrow \infty$.
\end{corollary}

 \begin{proof}[Proof of Proposition \ref{pll}]
 We are just proof the result of ${(\dot KE_{\vec{u},\vec{v}}^{\beta,s})_t(\rn)}$.
 By Lemma \ref{LL},  \ref{Lllll}, and \ref{LLll}, we find that the space ${(\dot KE_{\vec{u},\vec{v}}^{\beta,s})_t(\rn)}$ satisfies (1), (2), and (3) of Definition 3.1. So, only need to check (4) of Definition 3.1. Observe that, there exist a cube $B\left(\mathbf{0}, 2^{r}\right)\in \mathbb{B}$ with $r \in (0,\infty)$ such that $B \subset B\left(\mathbf{0}, 2^{r}\right)$. By this and Lemma \ref{Lllll}, we have
 $$
 \lf\|\mathbf1_{B(x,r)}\r\|_{(\dot KE_{\vec{u},\vec{v}}^{\beta,s})_t(\rn)}
 \leq
 \lf\|\mathbf1_{B(0,2^{r})}\r\|_{(\dot KE_{\vec{u},\vec{v}}^{\beta,s})_t(\rn)}.
 $$
 We estimate $\lf\|\mathbf1_{B(0,2^{r})}\r\|_{(\dot KE_{\vec{u},\vec{v}}^{\beta,s})_t(\rn)}$.
 If $s \in(0, \infty)$, then, by $\beta \in\left(-\sum_{i=1}^n1/u_{i}, \infty\right)$, we conclude that
\begin{align*}
 \lf\|\mathbf1_{B(0,2^{r})}\r\|_{(\dot KE_{\vec{u},\vec{v}}^{\beta,s})_t(\rn)}
&=\left[\sum_{k \in \mathbb{Z}} 2^{ks\beta}\left\|\mathbf{1}_{B(0,2^{r})} \mathbf{1}_{B_{k}}\right\|_{(E_{\vec{v}}^{\vec{v}})_t(\rn)}^{s}\right]^{\frac{1}{s}}\\
&\sim\left[\sum_{k=M} 2^{k s\left(\beta+\sum_{i=1}^n\frac{1}{u_{i}}\right)}\right]^{\frac{1}{s}}<\infty,
\end{align*}

where $M=[-\infty,r]^{n}$. By this, we obtain
$\lf\|\mathbf1_{B(x,r)}\r\|_{(\dot KE_{\vec{u},\vec{v}}^{\beta,s})_t(\rn)}<\infty$, thus,
$(\dot KE_{\vec{u},\vec{v}}^{\beta,s})_t(\rn)$ satisfies (4) of Definition \ref{defl}, which completes the proof of sufficiency.

In what follows, we to prove the necessity. When $\beta \in\left(-\infty,-\sum_{j=1}^n\frac{1}{q_j}\right]$ with $j \in\{1, \ldots, n\}$, there exists a $l \in \mathbb{Z} \cap(-\infty, 0]$ such that $B(\mathbf{0}, 1) \supset B\left(\mathbf{0}, 2^{l}\right)$. Using Lemma \ref{Lllll}, we conclude that
$$
\left\|\mathbf{1}_{B\left(\mathbf{0}, 2^{l}\right)}\right\|_{(\dot KE_{\vec{u},\vec{v}}^{\beta,s})_t(\rn)}
\leq
\left\|\mathbf{1}_{B(\mathbf{0}, 1)}\right\|_{(\dot KE_{\vec{v},\vec{u}}^{\beta,s})_t(\rn)}.
$$
Namely, $\mathbf{1}_{B(\mathbf{0}, 1)} \notin (\dot KE_{\vec{u},\vec{v}}^{\beta,s})_t(\rn)$ since the left-hand side of this inequality is infinity. Thus, $(\dot KE_{\vec{u},\vec{v}}^{\beta,s})_t(\rn)$ is not a ball quasi-Banach function space. We get what we wanted.
\end{proof}

\begin{proposition}
Let $\beta\in\rr$, $t,s\in(0,\fz)$ and $\vec{v}$, $\vec{u} \in(1,\fz)^{n}$. The following conclusion are correct:
\begin{enumerate}
  \item[(1)] if $s_1\leq s_2$, then $(\dot KE_{\vec{u},\vec{v}}^{\beta,s_1})_t(\rn)\subset
  (\dot KE_{\vec{u},\vec{v}}^{\beta,s_2})_t(\rn)$ and $ (KE_{\vec{u},\vec{v}}^{\beta,s_1})_t(\rn)
  \subset(KE_{\vec{u},\vec{v}}^{\beta,s_2})_t(\rn)$;
  \item[(2)] if $\beta_2\leq \beta_1$, then $(KE_{\vec{u},\vec{v}}^{\beta_1,s})_t(\rn)
  \subset(KE_{\vec{u},\vec{v}}^{\beta_2,s})_t(\rn)$;
  \item[(3)] if $\vec{u_{1}} \leqslant \vec{u_{2}}$, then $(\dot KE_{\vec{u_{2}},\vec{v}}^{\beta,s})_t(\rn) \subset (\dot KE_{\vec{u_{1}},\vec{v}}^{\beta,s})_t(\rn)$ and $( KE_{\vec{u_{2}},\vec{v}}^{\beta,s})_t(\rn) \subset (KE_{\vec{u_{1}},\vec{v}}^{\beta,s})_t(\rn)$;
\item[(4)] if $\vec{v_{1}} \leqslant \vec{v_{2}}$, then $(\dot KE_{\vec{u},\vec{v_{2}}}^{\beta,s})_t(\rn) \subset (\dot KE_{\vec{u},\vec{v_{1}}}^{\beta,s})_t(\rn)$ and $( KE_{\vec{u},\vec{v_{2}}}^{\beta,s})_t(\rn) \subset (KE_{\vec{u},\vec{v_{1}}}^{\beta,s})_t(\rn)$.

\end{enumerate}
\end{proposition}
\begin{proof}
Let's start with (1). It is easy to see that (1) is a consequence of the inequality in \cite{JK}.
\begin{equation}\label{cp}
\lf(\sum\limits_{m=1}^\fz|b_m|\r)^v\leq\sum\limits_{m=1}^\infty|b_m|^v, ~~~~if~ 0<v<1.
\end{equation}
Then through this inequality, we got
$$
\begin{aligned}
\|f\|_{(KE_{\vec{u},\vec{v}}^{\beta,s_2})_t(\rn)}
&=\left(\sum_{k \in \mathbb{Z}}(2^{k \beta}
\left\|f \mathbf1_{k}\right\|_{(E^{\vec{u}}_{\vec{v}})_t(\rn)})^{s_{2}}\right)^{\frac{1}{s_{1}} \frac{s_{1}}{s_{2}}}\\
 &\leqslant\left(\sum_{k \in \mathbb{Z}}(2^{k \beta}\|f \mathbf1_{k}\|_{(E^{\vec{u}}_{\vec{v}})_t(\rn)})^{s_{1}}\right)^{\frac{1}{s_{1}}}
\leqslant\|f\|_{(KE_{\vec{u},\vec{v}}^{\beta,s_1})_t(\rn)}.
\end{aligned}
$$

We remark that, we can get the (2) and (3) immediately via H\"{o}lder's inequality. In what follows, we show proof of (4)
$$
\begin{aligned}
&\quad\|f\mathbf1_{B(\cdot,t)}\|_{L^{\vec{v_{1}}}}\\
&=\left(\int_{\mathbb{R}}
\ldots\left(\int_{\mathbb{R}}\left(\int_{\mathbb{R}}
\left|f\mathbf1_{B(\cdot,t)}\right|^{v_{11}}
d x_1\right)^{\frac{v_{12}}{v_{11}}} d x_2\right)^{\frac{v_{13}}{v_{12}}} \ldots d x_n\right)^{\frac{1}{v_{1n}}}\\
&\leqslant\left(\int_{2^{k-1} \leqslant\left|x_n\right|<2^k} \ldots\left(\int_{2^{k-1} \leqslant\left|x_1\right|<2^k}|f(x)\mathbf1_{B(\cdot,t)(x)}|^{v_{11}} d x_1\right)^{\frac{v_{12}}{v_{11}}} \ldots d x_n\right)^{\frac{1}{v_{1n}}} \\
&\leq2^{k\sum_{i=1}^n\frac{1}{u_{1i}}-\sum_{i=1}^n\frac{1}{u_{2i}}}
\left(\int_{\mathbb{R}}\ldots
\left(\int_{\mathbb{R}}\left(\int_{\mathbb{R}}
\left|f\mathbf1_{B(\cdot,t)}\right|^{v_{21}}
d x_1\right)^{\frac{v_{22}}{v_{21}}} d x_2\right)^{\frac{v_{23}}{v_{22}}} \ldots d x_n\right)^{\frac{1}{v_{2n}}}\\
&\sim\frac{\|\mathbf1_{B(\cdot,t)}\|_{L^{\vec{v_{1}}}(\mathbb{R}^n)}}
{\|\mathbf1_{B(\cdot,t)}\|_{L^{\vec{v_{2}}}(\mathbb{R}^n)}}\|f\mathbf1_{B(\cdot,t)}\|_{L^{\vec{v_{2}}}}.
\end{aligned}
$$
Thus,
$$
\left\|\frac{\|f\mathbf1_{B(\cdot,t)}\|_{L^{\vec{v_{1}}}
(\mathbb{R}^n)}}{\|\mathbf1_{B(\cdot,t)}\|_{L^{\vec{v_{1}}}(\mathbb{R}^n)}}\right\|_{L^{\vec{u}}}
\leq
\left\|\frac{\|f\mathbf1_{B(\cdot,t)}\|_{L^{\vec{v_{2}}}
(\mathbb{R}^n)}}{\|\mathbf1_{B(\cdot,t)}\|_{L^{\vec{v_{2}}}(\mathbb{R}^n)}}\right\|_{L^{\vec{u}}}.
$$
this, together with Definition 2.1, we obtain
$$
\|f\|_{(\dot KE_{\vec{u},\vec{v_{1}}}^{\beta,s})_t(\rn)}
\leq
\|f\|_{(\dot KE_{\vec{u},\vec{v_{2}}}^{\beta,s})_t(\rn)}.
$$
Using a similar method can also get the result (1), (3), and (4) for non-homogeneous mixed-norm Herz-slice space.
\end{proof}
\section{Block decompositions }\label{B}
In this section we establish the decomposition characterizations of mixed-norm Herz-slice spaces.
\begin{definition}
 Let $\beta\in\rr$, $t,s\in(0, \infty)$ and $\vec{v}, \vec{u} \in(1,\fz)^{n}$.\\
(i) A function $\kappa(x)$ on $\rn$ is said to be a central $(\beta,u_{i},v_{i})$-block if
  \begin{enumerate}
    \item[(1)] $\supp(\kappa)\subset B(0,\lambda)$, for some $\lambda>0$;
    \item[(2)] $\|\kappa\|_{(E^{\vec{u}}_{\vec{v}})_t(\rn)}\leq C\lambda^{-\beta}$.
  \end{enumerate}
(ii) A function $\mu(x)$ on $\rn$ is said to be a central $(\beta,u_{i},v_{i})$-block of restrict type if
  \begin{enumerate}
    \item[(1)] $\supp(\mu)\subset B(0,\lambda)$ for some $\lambda\ge 1$;
    \item[(2)] $\|\mu\|_{(E^{\vec{u}}_{\vec{v}})_t(\rn)}\leq C\lambda^{-\beta}$.
  \end{enumerate}
If $\lambda=2^l$ for some $l\in\zz$ , then the corresponding central block is called a dyadic central block.
\end{definition}
\begin{theorem}\label{Thdf}
Let $\beta\in\rr$, $t,s\in(0, \infty)$ and $\vec{v}, \vec{u} \in(1,\fz)^{n}$. The following statements are equivalent:

(1) $\phi\in(\dot KE_{\vec{u},\vec{v}}^{\beta,s})_t(\rn)$;

(2) $\phi$ be able to present as
  \begin{equation}\label{df}
  \phi(x)=\sum\limits_{l\in\zz}\eta_l\mu_l(x),
  \end{equation}
where $\sum\limits_{k\in\zz}|\eta_l|^s<\fz$ and each $\mu_l$ is a dyadic central $(\beta,u_{i},v_{i})$-block with support contained in $B_l$.
\end{theorem}
\begin{proof}
For $\phi\in(\dot KE_{\vec{u},\vec{v}}^{\beta,s})_t(\rn)$, write
\begin{align*}
\phi(x)
&=\sum\limits_{l\in\zz}
\frac{|B_l|^{\frac{\beta}{n}}\lf\|\phi\mathbf1_{S_l}\r\|_{(E^{\vec{u}}_{\vec{v}})_t(\rn)}
                   \phi(x)\mathbf1_{S_l}(x)}{|B_l|^{\frac{\beta}{n}}\lf\|\phi\mathbf1_{S_l}\r\|_{(E^{\vec{u}}_{\vec{v}})_t(\rn)}}.
\end{align*}
When
$$
\eta_l=|B_l|^{\frac{\beta}{n}}\lf\|\phi\mathbf1_{S_l}\r\|_{(E^{\vec{u}}_{\vec{v}})_t(\rn)}
~~\text{and}~~\mu_l(x)=\frac{\phi(x)\mathbf1_{S_l}(x)}{|B_l|^{\frac{\beta}{n}}\lf\|\phi\mathbf1_{S_l}\r\|_{(E^{\vec{u}}_{\vec{v}})_t(\rn)}},
$$
it suffices to show that, $\supp(\mu_l)\subset B_l$, $\|\mu_l\|_{(E^{\vec{u}}_{\vec{v}})_t(\rn)}=|B_l|^{-\frac{\beta}{n}}$, and 
$\phi(x)=\sum\limits_{l\in\zz}\eta_l\mu_l(x)$. Therefore, each $\mu_l$ is a dyadic central $(\beta,u_{i},v_{i})$-block with the support $B_l$ and
$$
\sum\limits_{l\in\zz}|\lz_l|^s
=\sum\limits_{l\in\zz}|B_l|^{\frac{\beta s}n}\lf\|\phi\mathbf1_{S_l}\r\|^s_{(E^{\vec{u}}_{\vec{v}})_t(\rn)}
=\|\phi\|^s_{(\dot KE_{\vec{u},\vec{v}}^{\beta,s})_t(\rn)}<\fz.
$$
It remains to be shown that other side. Let $\phi(x)=\sum\limits_{l\in\zz}\eta_l\mu_l(x)$ be a decomposition of $\phi$. For each $m\in\zz$, we have
\begin{equation}\label{edf}
\lf\|\phi\mathbf1_{S_m}\r\|_{(E^{\vec{u}}_{\vec{v}})_t(\rn)}\leq \sum\limits_{l\ge m}|\eta_l|\lf\|\mu_l\r\|_{(E^{\vec{u}}_{\vec{v}})_t(\rn)}.
\end{equation}
Thus, if $0<s\leq 1$
\begin{align*}
\|\phi\|^s_{(\dot KE_{\vec{u},\vec{v}}^{\beta,s})_t(\rn)}
&=\sum\limits_{l\in\zz}2^{l\beta s}\lf\|\phi\mathbf1_{S_l}\r\|^s_{(E_{\vec{u}}^{\vec{v}})_t(\rn)}
\leq \sum\limits_{l\in\zz}2^{l\beta s}\lf(\sum\limits_{m\ge l}|\eta_l|^s\|\mu_m\|^s_{(E_{\vec{u}}^{\vec{v}})_t(\rn)}\r)\\
&\leq \sum\limits_{l\in\zz}2^{l\beta s}\lf(\sum\limits_{m\ge l}|\eta_m|^s2^{\beta ms}\r)
\leq C\sum\limits_{l\in\zz}|\eta_l|^s<\fz.
\end{align*}
If $1<s<\fz$, based on H\"older's inequality and (\ref{edf}),
\begin{align*}
\lf\|\phi\mathbf1_{S_m}\r\|_{(E^{\vec{u}}_{\vec{v}})_t(\rn)}
&\leq\sum\limits_{l\ge m}|\eta_l|\lf\|\mu_l\r\|^\frac12_{(E^{\vec{u}}_{\vec{v}})_t(\rn)}\lf\|\mu_l\r\|^\frac12_{(E^{\vec{u}}_{\vec{v}})_t(\rn)}\\
&\leq \lf(\sum\limits_{l\ge m}|\mu_l|^s\|\mu_l\|^{\frac s2}_{(E^{\vec{u}}_{\vec{v}})_t(\rn)}\r)^\frac1s
\lf(\sum\limits_{l\ge m}\|\mu_l\|^{\frac{s'}2}_{(E^{\vec{u}}_{\vec{v}})_t(\rn)}\r)^{\frac1{s'}}\\
&\leq \lf(\sum\limits_{l\ge m}|\eta_l|^s2^{-\frac {\beta ls}2}\r)^\frac1s
\lf(\sum\limits_{l\ge m}2^{-\frac{\beta ls'}2}\r)^{\frac1{s'}}.
\end{align*}
Therefore,
\begin{align*}
\lf\|\phi\r\|^s_{(\dot KE_{\vec{u},\vec{v}}^{\beta,s})_t(\rn)}
&\leq C\sum_{m\in\zz}2^{\beta ms}
\lf(\sum\limits_{l\ge m}|\eta_l|^s2^{-\frac {\beta l s}2}\r)^\frac1s
\lf(\sum\limits_{l\ge m}2^{-\frac{\beta ls'}2}\r)^{\frac1{s'}}\\
&\leq C\sum_{m\in\zz}|\eta_l|^s\sum\limits_{m\leq l}2^{\beta (m-l)s/2}
\leq C\sum_{l\in\zz}|\eta_l|^s<\fz.
\end{align*}
This completes the conclusion that we want.
\end{proof}
\begin{remark}
By using Theorem \ref{Thdf}, we can show that
$$
\|\phi\|_{(\dot KE_{\vec{u},\vec{v}}^{\beta,s})_t(\rn)}\sim\lf(\sum_{l\in\zz}|\eta_l|^s\r)^\frac1s.
$$
\end{remark}
A similar result is given for $(KE_{\vec{u},\vec{v}}^{\beta,s})_t(\rn)$ as follows, we acknowledge the details.
\begin{theorem}\label{Thndf}
Let $\beta\in\rr$, $t,s\in(0, \infty)$ and $\vec{v}, \vec{u} \in(1,\fz)^{n}$. The following statements are equivalent:

(1)  $\phi\in( KE_{\vec{u},\vec{v}}^{\beta,s})_t(\rn)$;

(2) $\phi$ be able to present as
  \begin{equation}\label{ndf}
  \phi(x)=\sum\limits_{l=0}^\fz\eta_l\mu_l(x),
  \end{equation}
where $\sum\limits_{l=0}^\fz|\eta_l|^s<\fz$ and each $\mu_l$ is a dyadic central $(\beta,u_{i},v_{i})$-block of restrict type
  with support contained in $B_l$.
Moreover,
$$
\|\phi\|_{(KE_{\vec{u},\vec{v}}^{\beta,s})_t(\rn)}\sim\lf(\sum_{l\ge 0}|\eta_l|^s\r)^\frac1s.
$$
\end{theorem}

\section{Boundedness of the Hardy–Littlewood maximal operator}\label{M}
The aim of this section is to give the boundedness of the Hardy–Littlewood maximal operator $\mathscr{M}$ on $(\dot KE_{\vec{u}, \vec{v}}^{\beta,s})_t(\rn)$ and $(KE_{\vec{u}, \vec{v}}^{\beta,s})_t(\rn)$.
First, we show that $\mathscr{M}$ is well defined on $(\dot KE_{\vec{u}, \vec{v}}^{\beta,s})_t(\rn)$. 
\begin{lemma}\label{zzzz}
Let $\beta\in\rr$ and $t,s \in (0, \infty)$. Let $\vec{v}, \vec{u} \in(1,\fz)^{n}$. If $\beta\in\left(-\sum_{i=1}^n{1}/{u_i}, \infty\right)$ with $i \in\{1, \ldots, n\}$. Then, for any $\phi\in (\dot KE_{\vec{u}, \vec{v}}^{\beta,s})_t(\rn)$, we have
$$
(\dot KE_{\vec{u}, \vec{v}}^{\beta,s})_t(\rn) \subseteq L_{\mathrm{loc}}^1\left(\mathbb{R}^n\right).
$$
\end{lemma}

\begin{proof} 
For any $\phi \in (\dot KE_{\vec{u}, \vec{v}}^{\beta,s})_t(\rn)$ and ball $B\left(x_0, \lambda\right)$ with $x_0 \in \mathbb{R}^n$ and $\lambda\in(0, \infty)$, by Lemma \ref{lll} and Corollary \ref{NN}, we have
\begin{align*}
\left\|\phi \mathbf1_{B\left(x_0, \lambda\right)}\right\|_{L^{1}}
&\leq\left\|\phi \mathbf1_{B\left(x_0, \lambda\right)}\right\|_{(\dot KE_{\vec{u}, \vec{v}}^{\beta,s})_t(\rn)}\left\|\mathbf1_{B\left(x_0, \lambda\right)}\right\|
_{(\dot KE_{\vec{u'}, \vec{v'}}^{\beta^{'},s^{'}})_t(\rn)}\\
&\leq\left\|\phi\right\|_{(\dot KE_{\vec{u}, \vec{v}}^{\beta,s})_t(\rn)}\left\|\mathbf1_{B\left(x_0, \lambda\right)}\right\|
_{(\dot KE_{\vec{u}, \vec{v}}^{\beta,s})_t(\rn)}.
\end{align*}
Due to $\mathbf1_{B\left(x_0, \lambda\right)}\in (\dot KE_{\vec{u}, \vec{v}}^{\beta,s})_t(\rn)$ by Proposition \ref{pll}. This proves the required conclusion.
\end{proof}

\begin{theorem}
Let $\beta\in\rr$, $t,s\in (0, \infty)$, $\vec{v}, \vec{u} \in(1,\fz)^{n}$ and $-\sum_{i=1}^n1/u_{i}<\beta<n-1/\sum_{i=1}^n1/u_{i}$. Suppose that the $\mathscr{M}$
satisfies

(1) for suitable function $\phi$ with $\supp(\phi)\subset S_k$ and $|x|\ge 2^{k+1}$ with $k\in\zz$,
\begin{equation}\label{T1}
|\mathscr{M}\phi(x)|\leq C\|\phi\|_{L^1(\rn)}|x|^{-n};
\end{equation}

(2) for suitable function $\phi$ with $\supp(\phi)\subset S_k$ and $|x|\le 2^{k-2}$ with $k\in\zz$,
\begin{equation}\label{T2}
|\mathscr{M}\phi(x)|\leq C2^{-kn}\|\phi\|_{L^1(\rn)}.
\end{equation}For any $\phi\in L_{\mathrm{loc}}^1\left(\mathbb{R}^n\right)$, then $\mathscr{M}\phi$ is bounded on $(\dot KE_{\vec{u}, \vec{v}}^{\beta,s})_t(\rn)$.
\end{theorem}
\begin{proof}
Let
$$\phi(x)=\sum\limits_{m\in\zz} \phi(x)\mathbf1_{S_m}(x):= \sum\limits_{m\in\zz} \phi_m(x).$$
By Lemma \ref{zzzz}, we know $\mathscr{M}\phi$ is well defined on $(\dot KE_{\vec{u},\vec{v}}^{\beta,s})_t(\rn)$. Observe that
$$
\begin{aligned}
\lf\|\mathscr{M}\phi\r\|_{(\dot KE_{\vec{u},\vec{v}}^{\beta,s})_t(\rn)}
&=\lf(\sum\limits_{k\in\zz}2^{k\beta s}\lf\|\mathscr{M}\phi\mathbf1_{S_k}\r\|^s_{(E_{\vec{v}}^{\vec{u}})_t(\rn)}\r)^\frac1s\\
&=\lf(\sum\limits_{k\in\zz}2^{k\beta s}\lf\|\sum\limits_{m=-\fz}^\fz \mathscr{M}\phi_m\mathbf1_{S_k}\r\|^s_{(E_{\vec{v}}^{\vec{u}})_t(\rn)}\r)^\frac1s
\end{aligned}
$$
$$
\begin{aligned}
&\leq \lf(\sum\limits_{k\in\zz}2^{k\beta s}\lf\|\sum\limits_{m=-\fz}^{k-2} \mathscr{M}\phi_m\mathbf1_{S_k}\r\|^s_{(E_{\vec{v}}^{\vec{u}})_t(\rn)}\r)^\frac1s\\
&+\lf(\sum\limits_{k\in\zz}2^{k\beta s}\lf\|\sum\limits_{m=k-1}^{k+1} \mathscr{M}\phi_m\mathbf1_{S_k}\r\|^s_{(E_{\vec{v}}^{\vec{u}})_t(\rn)}\r)^\frac1s\\
&+\lf(\sum\limits_{k\in\zz}2^{k\beta s}\lf\|\sum\limits_{m=k+2}^\fz \mathscr{M}\phi_m\mathbf1_{S_k}\r\|^s_{(E_{\vec{v}}^{\vec{u}})_t(\rn)}\r)^\frac1s
:=\mathrm{I}+\mathrm{II}+\mathrm{III}.
\end{aligned}
$$
For $\mathrm{I}$, from Lemma \ref{l} and (\ref{T2}), we find that
$$
\begin{aligned}
&\lf\|\sum\limits_{m=-\fz}^{k-2} \mathscr{M}\phi_m\mathbf1_{S_k}\r\|^s_{(E_{\vec{v}}^{\vec{u}})_t(\rn)}\\
&\lesssim\lf\|\sum\limits_{m=-\fz}^{k-2} \lf\|\phi_m\r\|_{L^1(\rn)}2^{-kn}\mathbf1_{S_k}\r\|^s_{(E_{\vec{v}}^{\vec{u}})_t(\rn)}\\
&\lesssim\lf\|\sum\limits_{m=-\fz}^{k-2} \lf\|\phi_m\r\|_{(E_{\vec{v}}^{\vec{u}})_t(\rn)}\lf\|\mathbf1_{S_m}\r\|_{(E_{{\vec{v}}'}^{{\vec{u}}'})_t(\rn)}
2^{-kn}\mathbf1_{S_k}\r\|^s_{(E_{\vec{v}}^{\vec{u}})_t(\rn)}.
\end{aligned}
$$
For $s\in(0,1]$, by Lemma \ref{lll}, we have
$$
\begin{aligned}
\mathrm{I}&
\lesssim\lf(\sum\limits_{k\in\zz}2^{k\beta s}\sum\limits_{m=-\fz}^{k-2} \|\phi_m\|^s_{(E_{\vec{v}}^{\vec{u}})_t(\rn)}
2^{(m-k)s(n-\sum_{i=1}^n\frac{1}{u_i})}\r)^\frac1s\\
&\lesssim \lf(\sum\limits_{m\in\zz}2^{m\beta s}\|\phi_m\|^s_{(E_{\vec{v}}^{\vec{u}})_t(\rn)}\r)^\frac1s
\lesssim \|\phi\|_{(\dot KE_{\vec{u},\vec{v}}^{\beta,s})_t(\rn)},
\end{aligned}
$$
for $s\in(1,\fz)$, by Lemma \ref{l} and (\ref{T2}), we have
$$
\begin{aligned}
\lf\|\sum\limits_{m=-\fz}^{k-2} \mathscr{M}\phi_m\mathbf1_{S_k}\r\|^s_{(E_{\vec{v}}^{\vec{u}})_t(\rn)}
\lesssim\lf\|\sum\limits_{m=-\fz}^{k-2} \|\phi_m\|_{L^1(\rn)}2^{-kn}\mathbf1_{S_k}\r\|^s_{(E_{\vec{v}}^{\vec{u}})_t(\rn)}
\end{aligned}
$$
$$
\begin{aligned}
&\lesssim\lf\|\sum\limits_{m=-\fz}^{k-2} \|\phi_m\|_{(E_{\vec{v}}^{\vec{u}})_t(\rn)}
\lf\|\mathbf1_{S_k}\r\|_{(E_{{\vec{v'}}}^{{\vec{u'}}})_t(\rn)}2^{-kn}\mathbf1_{S_k}\r\|^s_{(E_{\vec{v}}^{\vec{u}})_t(\rn)}\\
&\lesssim \lf(\sum\limits_{m=-\fz}^{k-2} \|\phi_m\|_{(E_{\vec{v}}^{\vec{u}})_t(\rn)}\r)^s\lf\|\mathbf1_{S_m}\r\|^s_{(E_{{\vec{v'}}}^{{\vec{u'}}})_t(\rn)}2^{-kns}
\lf\|\mathbf1_{S_k}\r\|^s_{(E_{\vec{v}}^{\vec{u}})_t(\rn)}.\\
\end{aligned}
$$
By Lemma \ref{lll}, we deduce
$$
\begin{aligned}
\mathrm{I}
&\lesssim \lf(\sum\limits_{k\in\zz}2^{k\beta s}\lf(\sum\limits_{m=-\fz}^{k-2} \|\phi_m\|_{(E_{\vec{v}}^{\vec{u}})_t(\rn)}\r)^s2^{(k-m)s(\sum_{i=1}^n\frac{1}{u_i}-n)}\r)^\frac1s\\
&\lesssim \lf(\sum\limits_{k\in\zz}2^{k\beta s}\lf(\sum\limits_{m=-\fz}^{k-2} \|\phi_m\|^s_{(E_{\vec{v}}^{\vec{u}})_t(\rn)}2^{(k-m)(\sum_{i=1}^n\frac{1}{u_i}-n)s/2}\r)\r)^\frac1s\\
&\hs\times\lf(\lf(\sum\limits_{m=-\fz}^{k-2}2^{(k-m)(\sum_{i=1}^n\frac{1}{u_i}-n)s'/2}\r)^{s/s'}\r)^\frac1s
\lesssim \|\phi\|_{(\dot KE_{\vec{u},\vec{v}}^{\beta,s})_t(\rn)}.
\end{aligned}
$$
We remark that $\mathscr{M}$ is bounded on $(E^{\vec{u}}_{\vec{v}})_t(\mathbb{R}^n)$ \cite[Lemma 2.5]{LZWO}, for $\mathrm{II}$, we see that
$$
\begin{aligned}
\mathrm{II}
&\lesssim \lf(\sum\limits_{k\in\zz}2^{k\beta s}\lf\|\sum\limits_{m=k-1}^{k+1} \mathscr{M}\phi_m\r\|^s_{(E_{\vec{v}}^{\vec{u}})_t(\rn)}\r)^\frac1s\\
&\lesssim \lf(\sum\limits_{k\in\zz}\sum\limits_{m=k-1}^{k+1}2^{(k-m)\beta s}2^{m\beta s}\|\phi_m\|^s_{(E_{\vec{v}}^{\vec{u}})_t(\rn)}\r)^\frac1s\\
&\lesssim \lf(\sum\limits_{m\in\zz}2^{m\beta s}\|\phi_m\|^s_{(E_{\vec{v}}^{\vec{u}})_t(\rn)}\r)^\frac1s
\lesssim \|\phi\|_{(\dot KE_{\vec{u},\vec{v}}^{\beta,s})_t(\rn)}.
\end{aligned}
$$
For the part of $\mathrm{III}$, using Lemma \ref{l} and (5.1), we see that
$$
\begin{aligned}
&\lf\|\sum\limits_{m=k+2}^\fz \mathscr{M}\phi_m\mathbf1_{S_k}\r\|^s_{(E_{\vec{v}}^{\vec{u}})_t(\rn)}
\lesssim \lf\|\sum\limits_{m=k+2}^\fz
\|\phi_m\|_{L^1(\rn)}2^{-mn}\mathbf1_{S_k}\r\|^s_{(E_{\vec{v}}^{\vec{u}})_t(\rn)}
\end{aligned}
$$
$$
\begin{aligned}
&\lesssim \lf\|\sum\limits_{m=k+2}^\fz 2^{-mn}
\mathbf1_{S_k}\|\phi_m\|_{(E_{\vec{v}}^{\vec{u}})_t(\rn)}
\lf\|\mathbf1_{S_m}\r\|_{(E_{{\vec{v'}}}^{{\vec{u'}}})_t(\rn)}\r\|^s_{(E_{\vec{v}}^{\vec{u}})_t(\rn)}\\
&\lesssim \sum\limits_{m=k+2}^\fz 2^{-mn}
\|\phi_m\|_{(E_{\vec{v}}^{\vec{u}})_t(\rn)}\|\mathbf1_{S_k}\|_{(E_{\vec{v}}^{\vec{u}})_t(\rn)}
\lf\|\mathbf1_{S_m}\r\|_{(E_{{\vec{v'}}}^{{\vec{u'}}})_t(\rn)}.
\end{aligned}
$$
When $0<s\leq 1$, by Lemma \ref{lll}, we can write
$$
\begin{aligned}
\mathrm{III}
&\lesssim \lf[\sum\limits_{k\in\zz}2^{k\beta s}\sum\limits_{m=k+2}^\fz \lf(2^{-mn}\|\phi_m\|_{(E_{\vec{v}}^{\vec{u}})_t(\rn)}
\lf\|\mathbf1_{B_m}\r\|_{(E_{{\vec{v'}}}^{{\vec{u'}}})_t(\rn)}
\lf\|\mathbf1_{B_k}\r\|_{(E_{\vec{v}}^{\vec{u}})_t(\rn )}\r)^{s}\r]^\frac1s\\
&\lesssim \lf(\sum\limits_{k\in\zz}2^{k\beta s}\sum\limits_{m=k+2}^\fz\|\phi_m\|^s_{(E_{\vec{v}}^{\vec{u}})_t(\rn)}
2^{ms(n-\sum_{i=1}^n\frac{1}{u_i})}2^{-msn}2^{ks\sum_{i=1}^n\frac{1}{u_i}}\r)^\frac1s\\
&\lesssim \lf(\sum\limits_{k\in\zz}2^{k\beta s}\sum\limits_{m=k+2}^\fz \lf(2^{m\beta-k\beta}\|\phi_m\|_{(E_{\vec{v}}^{\vec{u}})_t(\rn)}\r)^{s}
\r)^\frac1s
\lesssim \|\phi\|_{(\dot KE_{\vec{u},\vec{v}}^{\beta,s})_t(\rn)}.
\end{aligned}
$$
Using Lemma \ref{l} and \ref{T1}, we know that
$$
\begin{aligned}
&\lf\|\sum\limits_{m=k+2}^\fz \mathscr{M}\phi_m\mathbf1_{S_k}\r\|^s_{(E_{\vec{v}}^{\vec{u}})_t(\rn)}\\
&\lesssim \lf\|\sum\limits_{m=k+2}^\fz \|\phi_m\|_{L^1(\rn)}2^{-mn}\mathbf1_{S_k}\r\|^s_{(E_{\vec{v}}^{\vec{u}})_t(\rn)}\\
&\lesssim \lf\|\sum\limits_{m=k+2}^\fz \|\phi_m\|_{(E_{\vec{v}}^{\vec{u}})_t(\rn)}\lf\|\mathbf1_{S_m}\r\|_{(E_{{\vec{u'}}}^{{\vec{u'}}})_t(\rn)}2^{-mn}
\mathbf1_{S_k}\r\|^s_{(E_{\vec{v}}^{\vec{u}})_t(\rn)}.
\end{aligned}
$$
For $1<s<\fz$, by Lemma \ref{lll} yields
$$
\begin{aligned}
\mathrm{III}
&\lesssim \lf[\sum\limits_{k\in\zz}2^{k\beta s}\lf(2^{-mn}\sum\limits_{m=k+2}^\fz\|\phi_m\|_{(E_{\vec{v}}^{\vec{u}})_t(\rn)}
\lf\|\mathbf1_{B_m}\r\|_{(E_{{\vec{v'}}}^{{\vec{u'}}})_t(\rn)}
\lf\|\mathbf1_{B_k}\r\|_{(E_{\vec{v}}^{\vec{u}})_t(\rn)}\r)^{s}\r]^\frac1s\\
&\lesssim \lf(\sum\limits_{k\in\zz} 2^{m\beta s} \|\phi_m\|^s_{(E_{\vec{v}}^{\vec{u}})_t(\rn)}\r)^\frac1s
\lesssim \|\phi\|_{(\dot KE_{\vec{u},\vec{v}}^{\beta,s})_t(\rn)}.
\end{aligned}
$$
We got what we want.
\end{proof}

%
%
%
%
%



\hspace*{-0.6cm}\textbf{\bf Competing interests}\\
The authors declare that they have no competing interests.\\

\hspace*{-0.6cm}\textbf{\bf Funding}\\
The research was supported by the National Natural Science Foundation of China (Grant No. 12061069).\\

\hspace*{-0.6cm}\textbf{\bf Authors contributions}\\
All authors contributed equality and significantly in writing this paper. All authors read and approved the final manuscript.\\

\hspace*{-0.6cm}\textbf{\bf Acknowledgments}\\
All authors would like to express their thanks to the referees for valuable advice regarding previous version of this paper.\\

\begin{thebibliography}{10}

\bibitem{AG}
Aguilar-Hern\'{a}ndez, T.: 
Mixed norm spaces and RM(p, q) spaces. 
Rev. Mat. Complut. (2023) doi:10.1007/s13163-023-00466-1.
\bibitem{AI}
Antoni\'{c}, N., Ivec, I.: 
On the H\"{o}rmander-Mihlin theorem for mixed-norm Lebesgue spaces,.
J. Math. Anal. Appl. {\bf433}(1), 176-199 (2016)

\bibitem{AS}
Asekritova, I., Cerd\`{a}, J., Kruglyak, N.:
The Riesz-Herz equivalence for capacitary maximal functions. 
Rev. Mat. Complut. {\bf25}(1), 43-59 (2012)
Auscher, P., Prisuelos-Arribas, C.: 
Tent space boundedness via extrapolation.
Math. Z. 
{\bf286}(3-4), 1575-1604 (2017)

\bibitem{AP}
Auscher, P., Prisuelos-Arribas, C.:
Tent space boundedness via extrapolation.
Math. Z. 286(3-4), 1575-1604 (2017)
\bibitem{AM}
Auscher, P., Mourgoglou, M.:
Representation and uniqueness for boundary value elliptic problems via first order systems.
Rev. Mat. Iberoam. 
{\bf35}(1), 241-315 (2019)

\bibitem{BA}
Beurling, A.: 
Construction and analysis of some convolution algebras.
Ann. Inst. Fourier 
{\bf14}(2), 1-32 (1964) 
\bibitem{BP}
Benedek, A., Panzone, R.:
The space $L_{p}$, with mixed norm.
Duke Math. J. 
{\bf28}, 301-324 (1961)



\bibitem{CN}
Chikami, N.:
On Gagliardo-Nirenberg type inequalities in Fourier-Herz spaces. 
J. Funct.Anal. 
{\bf275}(5), 1138-1172 (2018) 




\bibitem{Herz}
Herz, C.S.:
Lipschitz spaces and Bernstein's theorem on absolutely convergent Fourier series.
J. Math. Mech. 
{\bf18}, 283-324 (1968)
\bibitem{HK}
Ho, K-P:
Extrapolation to Herz spaces with variable exponents and applications. 
Rev. Mat. Complut. {\bf33}(2), 437-463 (2020)


\bibitem{H}
H\"{o}rmander, L.:
Estimates for translation invariant operators in $L_{p}$ spaces.
Acta Math. 
{\bf104}, 93–140 (1960)

 

\bibitem{HFDY}
Huang, L., Weisz, F., Yang, D., Yuan, W.: 
Summability of fourier transforms on mixed-norm lebesgue spaces via associated herz spaces.
Anal. Appl. 
{\bf21}(1), 1-50 (2021)



\bibitem{JK}
Kuang, J.C.:
Applied Inequalities. Shandong Science and Technology Press. Jinan, (2004).
\bibitem{KN}
Krylov, N.V.: 
Parabolic equations with VMO coefficients in Sobolev spaces with mixed norms.
Funct. Anal. 
{\bf250}(2), 521-558 (2007)

\bibitem{LV}
Llinares, A., Vukoti\'{c}, D.:
Contractive inequalities for mixed norm spaces and the Beta function.
J. Math. Anal. Appl. 
{\bf509}(1), 125-938 (2022) 

\bibitem{LY}
Lu, S., Yang, D.:
The decomposition of the weighted Herz spaces on $\rn$ and its applications.
Sci. China Ser. A 
{\bf38}(2), 147–158 (1995)
\bibitem{LZW}
Lu, Y., Zhou, J., Wang, S.B.:
Herz-slice spaces and applications.
arXiv:2204.08635v1.
\bibitem{LZWO}
Lu, Y., Zhou, J., Wang, S.B.:
Operators on mixed-norm amalgam spaces via extrapolation. 
Preprint.
\bibitem{MWY}
Min, D., Wu, G., Yao, Z.:
Global well-posedness of strong solution to 2D MHD equations in critical Fourier-Herz spaces.
J. Math. Anal. Appl. 
{\bf504}(1), 125-345 (2021)

\bibitem{NZ}
Nie, Y., Zheng, X.:
III-posedness of the 3D in compressible hyperd issipative Navier-Stokes system in critical Fourier-Herz spaces.
Nonlinearity 
{\bf31}(7), 3115 (2018)


\bibitem{SHYY}
Sawano, Y., Ho, K-P., Yang, D.C., Yang, S.B.:
Hardy spaces for ball quasi-Banach function spaces.
Dissertationes Math. 
{\bf525}, 1-102 (2017)

\bibitem{ST}
Shi, Y.L., Tao, X.X.:
The boundedness of sublinear operators on Morrey-Herz spaces over the homogeneous type space.
Anal. Math. 
{\bf39}(1), 69-85 (2013) 




\bibitem{TYY}
Tao, J., Yang, D.C., Yuan, W., Zhang, Y.Y.:
Compactness characterizations of commutators on ball Banach function spaces.
Potential Anal. 
{\bf58}(4), 645–679  (2023)  

\bibitem{WM}
Wei, M.Q.:
A characterization of $CMO^{\vec{q}}(\rn)$ via the commutator of Hardy-type operators on mixed Herz spaces.
Appl. Anal. 
{\bf101}(16), 5727-5742 (2021)

\bibitem{WY}
Wei, M.Q., Yan, D.Y.:
Operators on Herz-type spaces associated with ball quasi-Banach function spaces.
arXiv:2209.04323.

\bibitem{ZZM}
Zhang, H., Zhou, J.:
Mixed-Norm Amalgam Spaces and Their Predual.
Symmetry.
{\bf14}(1), 74 (2022) 
\bibitem{ZYZ}
Zhao, Y., Yang, D.C., Zhang, Y.Y.:
Mixed-norm Herz spaces and their applications in related Hardy spaces. 
Anal. Appl.
(Singap.) (2022) 2250016, https://doi.org/10.1142/S0219530522500166.




\end{thebibliography}
\end{document}